\documentclass[reqno,12pt]{amsart} 
\usepackage{vmargin}
\setpapersize{A4}

\usepackage{amsmath} 

\usepackage{amsfonts}   

\usepackage{amssymb}      

\usepackage{eufrak}      

\usepackage{amscd}      
\usepackage{amsthm}      

\usepackage{tikz,amsmath}
\usetikzlibrary{arrows}

\usepackage{epsfig}      

\usepackage{amstext}      

\usepackage{graphicx}

\usepackage[all,line,dvips]{xy} 
\CompileMatrices 

\newcommand{\id}{\operatorname{id}} 
\newcommand{\Aut}{\operatorname{Aut}}

\newcommand{\Span}{\operatorname{Span}}
\newcommand{\Tr}{\operatorname{Tr}}

 \newcommand{\supp}{\operatorname{supp}}

\newcommand{\Var}{\operatorname{Var}}

   \theoremstyle{plain}
   \newtheorem{thm}{Theorem}[section]
   \newtheorem{prop}[thm]{Proposition}
   \newtheorem{lemma}[thm]{Lemma}  
   \newtheorem{cor}[thm]{Corollary}
   \theoremstyle{definition}

   \newtheorem{example}[thm]{Example}
   \theoremstyle{remark}
   \newtheorem{obs}[thm]{Observation}
   \newtheorem{remark}[thm]{Remark}

\usepackage{mdframed,xcolor}

\definecolor{mybgcolor}{gray}{0.8}
\definecolor{myframecolor}{rgb}{.647,.129,.149}

\usepackage{lipsum}

   \numberwithin{equation}{section}

        \date{\today}

\title[Diagonality]{Diagonality of actions
  and KMS weights}  
\author{Johannes Christensen and Klaus Thomsen}

\date{\today}

\email{matkt@math.au.dk   \qquad   johannes@math.au.dk}
\address{Department of Mathematics, Aarhus
  University, Ny Munkegade 118, 8000 Aarhus C, Denmark}

\begin{document}

\maketitle

\section{Introduction}

Recent years have seen an increasing interest in the investigation of
KMS states for one-parameter actions on $C^*$-algebras. While the
original motivation for the introduction of KMS states came from the
interpretation of these states as equilibrium states in models from
quantum statistical mechanics, the renewed interest stems also
from   
more purely mathematical considerations, where the KMS states have
been related to objects and structures from other fields, such as number theory or
dynamical systems. In the present paper we investigate relations
between properties of the KMS states and properties of 
the one-parameter action giving rise to them. As we shall
now explain, we show that the existence of a 'diagonal' KMS state or weight
implies that the action itself must be 'diagonal'.

For most if not all the one-parameter actions on
$C^*$-algebras for which we have been able to determine the structure
of the KMS states
 or KMS weights, the underlying $C^*$-algebra can be
 presented as the $C^*$-algebra of a locally compact groupoid, as
 introduced by Renault in \cite{Re1}, and the action described as
 arising from a continuous real-valued homomorphism on the groupoid by
 a canonical procedure also introduced in \cite{Re1}. For this reason the
 results of Neshveyev, \cite{N}, which extend results of Renault and give a general and abstract
 description of the KMS states for such actions on a groupoid
 $C^*$-algebra is of outmost importance.  In the following
 we call these
 actions \emph{diagonal}.

When the groupoid and the associated $C^*$-algebra is fixed, it is
certainly not all one-parameter actions that are diagonal. It follows
from Neshveyev's theorem, Theorem 1.3 in \cite{N}, that a diagonal action has the
property that if a KMS
state exists, there will also be one which factorises through the canonical
conditional expectation onto the abelian $C^*$-subalgebra generated by
the continuous compactly supported functions on the unit space. In the
following we call these states \emph{diagonal}. The present work
sprang from the realisation that in many cases the property that
there is a diagonal KMS state, characterises the diagonal actions. That
is, for many groupoid $C^*$-algebras a one-parameter action is
diagonal if and only if the action admits a diagonal KMS state. The
simplest example of this is perhaps the following.

Consider the $C^*$-algebra $M_n$ of complex $n$ by $n$ matrices. A
continuous one-parameter group $\alpha$ of automorphisms on $M_n$ is inner in
the sense that there is a self-adjoint matrix $A \in M_n$ such that
$$
\alpha_t(B) = e^{itA}Be^{-itA}
$$
for all $t \in \mathbb R$ and all $B \in M_n$. For each $\beta \in
\mathbb R$ there is a unique $\beta$-KMS state $\omega_{\beta}$ for
$\alpha$ given by
$$
\omega_{\beta}(B) = \frac{\Tr \left(e^{-\beta A} B\right)}{\Tr
  \left(e^{-\beta A}\right)} .
$$
It can be shown that for $\beta \neq 0$ the state $\omega_{\beta}$
factorises through the canonical (and unique) conditional expectation
from $M_n$ onto the $C^*$-subalgebra of diagonal matrices if and only
if $A$ is diagonal. It is this fact we will generalize. For this note
that $M_n$ is the groupoid $C^*$-algebra of the groupoid $\mathcal G = \{1,2,3, \cdots
,n\} \times \{1,2,3, \cdots,n \}$ with operations
$$
(a,b)(b,c) = (a,c) \ \text{and} \ (a,b)^{-1} = (b,a) .
$$
When $M_n$ is identified with the $C^*$-algebra $C^*(\mathcal G)$ of
$\mathcal G$, the diagonal matrices in $M_n$ constitute the
$C^*$-algebra $C\left(\mathcal G^{(0)}\right)$ of (continuous) functions on $\mathcal G$ whose support
is contained in the unit space
$$
\mathcal G^{(0)} = \left\{ (k,k) : \ k \in \{1,2,\cdots , n\} \right\}
$$
of $\mathcal G$. In this picture the conditional expectation onto the
diagonal matrices is the map
$$
P : C^*(\mathcal G) \to C\left(\mathcal G^{(0)}\right)
$$
which restricts functions to $\mathcal G^{(0)}$. Furthermore, the
matrix $A$ will be diagonal if and only if there is a groupoid
homomorphism $c : \mathcal G \to \mathbb R$ such that
\begin{equation}\label{intro1}
\alpha_t(f)(a,b) = e^{ic(a,b)t} f(a,b)
\end{equation}
for all $t \in \mathbb R, \ (a,b) \in \mathcal G$ and all $f \in
C^*(\mathcal G)$. Because the whole setup is so transparent in this
case, we can easily conclude that there is an equivalence between the
following conditions:
\begin{enumerate}
\item[1)] $\alpha$ is diagonal in the sense that there is a groupoid
homomorphism $c : \mathcal G \to \mathbb R$ such that (\ref{intro1})
holds.
\item[2)] There is a $\beta \neq 0$ and a $\beta$-KMS state
  $\omega_{\beta}$ for
  $\alpha$ which is diagonal in the sense that it factorises through
  the conditional expectation $C^*(\mathcal G) \to C\left(\mathcal
    G^{(0)}\right)$.
\item[3)] $\alpha_t(f) = f$ for all $t \in \mathbb R$ and all $f\in
  C\left(\mathcal G^{(0)}\right)$.
\item[4)] All $\beta$-KMS states of $\alpha$, for $\beta \neq 0$, are diagonal.
\end{enumerate}
Our main result is that these equivalences hold much more generally
as we shall now explain.

\section{Notation and main result}

Let $\mathcal G$ be a second countable locally compact Hausdorff \'etale
groupoid with unit
space $\mathcal G^{(0)}$. Let $r : \mathcal G \to \mathcal G^{(0)}$ and $s : \mathcal G \to \mathcal G^{(0)}$ be
the range and source maps, respectively. For $x \in \mathcal G^{(0)}$ put $\mathcal G^x = r^{-1}(x), \ \mathcal G_x = s^{-1}(x) \ \text{and} \ \mathcal G^x_x =
s^{-1}(x)\cap r^{-1}(x)$. Note that $\mathcal G^x_x$ is a group, the \emph{isotropy group} at $x$. The space $C_c(\mathcal G)$ of continuous compactly supported
functions is a $*$-algebra when the product is defined by
\begin{equation*}\label{k15}
(f_1 * f_2)(g) = \sum_{h \in \mathcal G^{r(g)}} f_1(h)f_2(h^{-1}g)
\end{equation*}
and the involution by $f^*(g) = \overline{f\left(g^{-1}\right)}$.
To define the \emph{reduced groupoid
  $C^*$-algebra} $C^*_r(\mathcal G)$, let $x\in \mathcal G^{(0)}$. There is a
representation $\pi_x$ of $C_c(\mathcal G)$ on the Hilbert space $l^2(\mathcal G_x)$ of
square-summable functions on $\mathcal G_x$ given by 
\begin{equation*}\label{pix}
\pi_x(f)\psi(g) = \sum_{h \in \mathcal G^{r(g)}} f(h)\psi(h^{-1}g) .
\end{equation*}
$C^*_r(\mathcal G)$ is the completion of $C_c(\mathcal G)$ with respect to the norm
$$
\left\|f\right\|_r = \sup_{x \in \mathcal G^{(0)}} \left\|\pi_x(f)\right\| .
$$  
Note that $C^*_r(\mathcal G)$ is separable since we assume that the
topology of $\mathcal G$ is second countable.

We shall here be concerned not only with KMS states, but more
generally with KMS weights. Let $A$ be a $C^*$-algebra and $A_{+}$ the convex cone of positive
elements in $A$. A \emph{weight} on $A$ is a map $\psi : A_{+} \to
[0,\infty]$ with the properties that $\psi(a+b) = \psi(a) + \psi(b)$
and $\psi(\lambda a) = \lambda \psi(a)$ for all $a, b \in A_{+}$ and all
$\lambda \in \mathbb R, \ \lambda > 0$. By definition $\psi$ is
\emph{densely defined} when $\left\{ a\in A_{+} : \ \psi(a) <
  \infty\right\}$ is dense in $A_{+}$ and \emph{lower semi-continuous}
when $\left\{ a \in A_{+} : \ \psi(a) \leq \alpha \right\}$ is closed
for all $\alpha \geq 0$. We will use \cite{Ku}, \cite{KV}  as our source for
information on weights, and we say that a weight is \emph{proper}
when it is non-zero, densely defined and lower semi-continuous. Let $\psi$ be a proper weight on $A$. Set $\mathcal N_{\psi} = \left\{ a \in A: \ \psi(a^*a) < \infty
\right\}$ and note that 
\begin{equation*}\label{f3}
\mathcal M_{\psi} = \Span \left\{ a^*b : \ a,b \in
  \mathcal N_{\psi} \right\}
\end{equation*} 
is a dense
$*$-subalgebra  of $A$, and that there is a unique well-defined linear
map $\mathcal M_{\psi} \to \mathbb C$ which
extends $\psi : \mathcal M_{\psi}\cap A_+ \to
[0,\infty)$. We denote also this densely defined linear map by $\psi$.

Let $\alpha : \mathbb R \to \Aut A$ be a continuous one-parameter group of automorphisms on
$A$. Let $\beta \in \mathbb R$. Following \cite{C} we say that a proper weight
$\psi$ on $A$ is a \emph{$\beta$-KMS
  weight} for $\alpha$ when
\begin{enumerate}
\item[i)] $\psi \circ \alpha_t = \psi$ for all $t \in \mathbb R$, and
\item[ii)] for every pair $a,b \in \mathcal N_{\psi} \cap \mathcal
  N_{\psi}^*$ there is a continuous and bounded function $F$ defined on
  the closed strip $D_{\beta}$ in $\mathbb C$ consisting of the numbers $z \in \mathbb C$
  whose imaginary part lies between $0$ and $\beta$, and is
  holomorphic in the interior of the strip and satisfies that
$$
F(t) = \psi(a\alpha_t(b)), \ F(t+i\beta) = \psi(\alpha_t(b)a)
$$
for all $t \in \mathbb R$. 
\end{enumerate}
Compared to \cite{C} we have changed the
orientation in order to have the same sign convention as in \cite{BR},
for example. It will be important for us that there is an alternative
characterisation of when a proper weight is a $\beta$-KMS
weight. Specifically, by Proposition 1.11 in \cite{KV} a proper weight
$\psi$ is a $\beta$-KMS weight for $\alpha$ if and only if it is
$\alpha$-invariant (as in i) above) and
\begin{equation}\label{not1}
\psi(a^*a) =
\psi\left(\alpha_{-\frac{i\beta}{2}}(a)\alpha_{-\frac{i\beta}{2}}(a)^*\right)
\end{equation}
for all $a$ in the domain $D\left(\alpha_{-\frac{i\beta}{2}}\right)$
of $\alpha_{-\frac{i\beta}{2}}$; the closure of the restriction of
$\alpha_{-\frac{i\beta}{2}}$ to the analytic elements for $\alpha$,
cf. \cite{Ku}. A $\beta$-KMS weight $\psi$ with the property that 
$$
\sup \left\{ \psi(a) : \ 0 \leq a \leq 1 \right\} = 1
$$
will be called a \emph{$\beta$-KMS state}.

Returning to the case $A = C^*_r(\mathcal G)$, note that the map $C_c(\mathcal G) \to C_c\left(\mathcal G^{(0)}\right)$ which
restricts functions to $\mathcal G^{(0)}$ extends to a conditional expectation $P
: C^*_r(\mathcal G) \to C_0\left(\mathcal G^{(0)}\right)$. Via $P$ a
regular Borel measure $m$ on $\mathcal G^{(0)}$ gives rise to a weight
$\varphi_m : C^*_r(\mathcal G)_+ \to [0,\infty]$ defined by the
formula
\begin{equation}\label{f1}
\varphi_m(a) = \int_{\mathcal G^{(0)}} P(a) \ dm .
\end{equation}
It follows from Fatou's lemma that $\varphi_{m}$ is lower
semi-continuous. Since $\varphi_{m}(f a f) < \infty$ for every
non-negative function $f$ in $C_c(\mathcal G^{(0)})$, it follows that $\varphi_m$
is also densely defined, i.e. $\varphi_m$ is a proper weight on
$C^*_r(\mathcal G)$ if and only if $m$ is not the zero measure. In the
following we say that a proper weight $\psi$ on $C^*_r(\mathcal G)$ is
\emph{diagonal} when $\psi = \varphi_m$ for some regular Borel measure
$m$ on $\mathcal G^{(0)}$. By the Riesz representation
theorem this occurs if and only if $\psi \circ P = \psi$.

Given a continuous homomorphism $c : \mathcal G \to
  \mathbb R$ there is a continuous one-parameter group $\sigma^c$ on $ C^*_r(\mathcal G)$ such that
\begin{equation}\label{44r}
\sigma^c_t(g)(\xi) = e^{it c(\xi)} g(\xi)
\end{equation}
for all $t\in \mathbb R$, all $g \in C_c(\mathcal G)$ and all $\xi \in
\mathcal G$, cf. \cite{Re1}. A one-parameter action of this kind will
be called \emph{diagonal} in the following. We can then
formulate our main result as follows.

\begin{thm}\label{iff} Let $\mathcal G$ be a locally compact second
  countable Hausdorff
  \'etale groupoid such that for at least one element $x \in \mathcal
  G^{(0)}$ the isotropy group $\mathcal G_x^x$ is trivial, i.e. $\mathcal
  G_x^x = \{x\}$, and that $\mathcal G$ is minimal in the sense that $s(r^{-1}(y))$  
is dense in $\mathcal G^{(0)}$ for all $y \in \mathcal
G^{(0)}$. Furthermore, assume that $\mathcal G^{(0)}$ is totally
disconnected.

Let $\alpha= (\alpha_t)_{t\in \mathbb R}$ be a continuous
  one-parameter group of automorphisms on $C^*_r(\mathcal G)$ and
  assume that for some $\beta_0 \neq 0$ there is a $\beta_0$-KMS
  weight for
  $\alpha$.

The following are equivalent:
\begin{enumerate}
\item[1)]  There is a $\beta_1 \neq 0$ and a diagonal $\beta_1$-KMS
  weight for $\alpha$.
\item[2)] Whenever $\beta \neq 0$ and there is a $\beta$-KMS weight
  for $\alpha$, there is also a diagonal $\beta$-KMS weight for $\alpha$. 
\item[3)] $\alpha_t(f) = f$ for all $t \in \mathbb R$ and all $f\in
  C_0\left(\mathcal G^{(0)}\right)$.
\item[4)] $\alpha_t\left(C_0\left(\mathcal G^{(0)}\right)\right) \subseteq
    C_0\left(\mathcal G^{(0)}\right)$ for all $t\in \mathbb R$. 
\item[5)] $\alpha$ is diagonal.
\end{enumerate}
\end{thm}

Some of the (non-trivial) implications hold with fewer assumptions. Specifically,
$1) \Rightarrow 3)$ holds without the assumption that the unit
space is totally disconnected by Proposition \ref{4}, and the
implication $3) \Rightarrow 5)$ holds assuming only that
the points with trivial isotropy are dense in $\mathcal G^{(0)}$
(i.e. if $\mathcal G$ is topologically principal) by Proposition
\ref{15}. The implication $5) \Rightarrow 2)$ holds whenever $\mathcal
G^{(0)}$ is totally disconnected, without any further assumptions, as it follows from Corollary
\ref{75}. It may be that this implication is true in general and if so
the theorem with 4) removed is true also when $\mathcal G^{(0)}$ is not totally
disconnected. However, the first two assumptions on $\mathcal G$ which are
equivalent to topological principality and minimality of $\mathcal G$
are certainly necessary for the implication $3) \Rightarrow 1)$ to
hold, cf. Example \ref{29}. Finally, the gauge action on the $C^*$-algebra of a
strongly connected (row-finite) graph with infinite Gurevich entropy does not admit
any KMS weights at all, cf. \cite{Th3}, showing that it is necessary to
assume the existence of some KMS-weight for the implication $5)
\Rightarrow 1)$ to hold.

\section{Neshveyev's theorem for KMS weights}

\begin{lemma}\label{53}Let $A$ be a $C^*$-algebra, $\alpha$ a
  continuous one-parameter group of automorphisms on $A$ and $\psi$ a KMS weight
  for $\alpha$. Let $p \in A$ be a projection in the fixed point
  algebra of $\alpha$. Then $\psi(p) <\infty$.
\end{lemma}
\begin{proof} Assume that $a \geq 0$, $\psi(a) < \infty$ and that
  $a^{\frac{1}{2}}$ is analytic for $\alpha$. Then Proposition 1.11 in
  \cite{KV} applies to conclude that
\begin{equation}\label{54}
 \psi\left(pap\right) =
 \psi\left(\alpha_{-\frac{i\beta}{2}}\left(a^{\frac{1}{2}}\right)
   p\alpha_{-\frac{i\beta}{2}}\left(a^{\frac{1}{2}}\right)^*\right)
 \leq
 \psi\left(\alpha_{-\frac{i\beta}{2}}\left(a^{\frac{1}{2}}\right)\alpha_{-\frac{i\beta}{2}}\left(a^{\frac{1}{2}}\right)^*\right) = \psi(a).
\end{equation} 
Let $\{b_k\}$ be a sequence of positive elements in $A$ such that
$\lim_{k \to \infty} b_k = p$ and $\psi(b_k) < \infty$ for all $k$. For each $n
\in \mathbb N$, set 
$$
c_{k,n} = \sqrt{\frac{n}{\pi}} \int_{\mathbb R} \alpha_t(b_k)
e^{-nt^2} \ dt  .
$$
Then $c_{k,n}$ is analytic for $\alpha$ and
$\psi\left(c_{k,n}^2\right) \leq
\left\|c_{k,n}\right\|\psi\left(c_{k,n}\right) \leq 
\left\|c_{k,n}\right\|\psi\left(b_k\right) < \infty$ for all $k,n$. It
follows therefore from (\ref{54}) that $\psi\left(pc_{k,n}^2p\right) \leq \psi\left(c_{k,n}^2\right)
< \infty$ for all $k,n$. Note that
$$
\lim_{k \to \infty} \lim_{n \to \infty} c_{k,n}^2 = \lim_{k \to
  \infty} b_k^2 = p^2 = p.
$$
It follows that there are $k,n$ such that $\left\|p -
  pc_{k,n}^2p\right\| \leq \frac{1}{2}$, and then spectral theory tells
us that $ pc_{k,n}^2p \geq \frac{1}{2} p$. Hence
$\psi(p) \leq 2 \psi\left(pc_{k,n}^2p\right) < \infty$.

\end{proof}

 Let $\mathcal G$ be a locally compact second
  countable Hausdorff
  \'etale groupoid and $c : \mathcal G \to \mathbb R$ a continuous
  homomorphism. Let $\mu$ be a regular Borel measure on $\mathcal
  G^{(0)}$ and $\beta \in \mathbb R$ a real number. We say
that $\mu$ is \emph{$(\mathcal G,c)$-conformal with exponent $\beta$},
as in \cite{Th3}, or that $\mu$ is \emph{quasi-invariant with
  Radon-Nikodym cocycle $e^{-\beta c}$}, as in \cite{N}, when 
\begin{equation}\label{conformal2}
\mu\left( s(W)\right) = \int_{r(W)} e^{\beta c\left( r_W^{-1}(x)\right)} \
  d\mu(x) 
\end{equation}
for every open bi-section $W \subseteq \mathcal G$, where $r^{-1}_W$
denotes the inverse $r : W \to r(W)$. For each $x\in \mathcal G^{(0)}$
we can consider the full group $C^*$-algebra $C^*\left(\mathcal
  G^x_x\right)$ of the discrete group $\mathcal G^x_x$, the isotropy
group at $x$. As in \cite{N} we denote for $g \in \mathcal
G^x_x$ by $u_g$ the characteristic function of the element $g$ when we
consider $C^*\left(\mathcal G^x_x\right)$ as a completion of
$C_c\left(\mathcal G^x_x\right)$. Thus
$u_g,  g \in \mathcal G^x_x$, are the canonical unitary generators of $C^*\left(\mathcal
  G^x_x\right)$.  Following \cite{N} we say that a
collection $\varphi_x, x \in \mathcal G^{(0)}$, of states on $C^*\left(\mathcal G^x_x\right)$ is a
\emph{$\mu$-measurable field} when the function
$$
\mathcal G^{(0)} \ni x \mapsto \sum_{g \in 
  \mathcal G^x_x} f(g) \varphi_x(u_g) 
$$
is $\mu$-measurable for all $f \in C_c(\mathcal G)$. We identify two
$\mu$-measurable fields $\left\{\varphi_x\right\}_{x \in \mathcal
  G^{(0)}}$ and $\left\{\varphi'_x\right\}_{x \in \mathcal
  G^{(0)}}$ when $\varphi_x = \varphi'_x$ for $\mu$-almost every $x$.

The following theorem is a version for weights of Theorem 1.3 in \cite{N}. Note
that it deals with the full groupoid $C^*$-algebra $C^*(\mathcal
G)$ which is an extension of the reduced groupoid $C^*_r(\mathcal
G)$. We refer to \cite{Re1} for the definition of the full groupoid
$C^*$-algebra. To understand the following theorem and its proof it
suffices to know that $C^*(\mathcal G)$, like $C^*_r(\mathcal G)$, is a completion of
$C_c(\mathcal G)$ and that a continuous homomorphism $c :\mathcal G \to
\mathbb R$ also defines a continuous one-parameter group $\sigma^c$ on
$C^*(\mathcal G)$ via the formula (\ref{44r}).

\begin{thm}\label{Nversion1} (Neshveyev's theorem for
  weights.) Let $\mathcal G$ be a locally compact second countable
  Hausdorff \'etale groupoid and let $c : \mathcal G \to \mathbb R$ be
  a continuous homomorphism. Assume that the unit space $\mathcal G^{(0)}$ of $\mathcal G$ is totally
  disconnected. 

There is a bijective correspondence
  between the $\beta$-KMS weights for $\sigma^c$ on $C^*(\mathcal G)$ and the pairs $\left(\mu,
  \left\{\varphi_x\right\}_{x \in \mathcal G^{(0)}}\right)$, where
$\mu$ is a regular Borel measure on $\mathcal G^{(0)}$ and
$\left\{\varphi_x\right\}_{x \in \mathcal G^{(0)}}$ is a
$\mu$-measurable field of states $\varphi_x$ on
$C^*(\mathcal G^x_x)$ such that
\begin{enumerate}
\item[a)] $\mu$ is quasi-invariant with Radon-Nikodym cocycle $e^{-\beta c}$,
\item[b)] $\varphi_x(u_g) = \varphi_{r(h)}\left(u_{hgh^{-1}}\right)$
    for $\mu$-almost every $x\in \mathcal G^{(0)}$ and all $g\in
    \mathcal G^x_x, \ h \in \mathcal G_x$, and
\item[c)] $\varphi_x(u_g)  = 0$ for $\mu$-almost every $x \in \mathcal
  G^{(0)}$ and all $g \in \mathcal G^x_x \backslash c^{-1}(0)$.
\end{enumerate}
The $\beta$-KMS weight $\phi$ corresponding to the pair $\left(\mu,
  \left\{\varphi_x\right\}_{x \in \mathcal G^{(0)}}\right)$ has the
properties that $C_c\left(\mathcal G\right) \subseteq \mathcal
M_{\phi}$ and
\begin{equation}\label{60}
\phi(f) = \int_{\mathcal G^{(0)}}  \sum_{g \in \mathcal G^x_x}
f(g)\varphi_x(u_g) \ d\mu(x) 
\end{equation}
when $f \in C_c(\mathcal G)$.
\end{thm}
\begin{proof} Let $\phi$ be a $\beta$-KMS weight for $\sigma^c$. Since
  $\mathcal G^{(0)}$ is totally disconnected by assumption there is a
  sequence $p_1 \leq p_2 \leq p_3 \leq \cdots$ of projections in
  $C_c(\mathcal G^{(0)})$ with the property that $\{p_n\}$ is an
  approximate unit for $C^*(\mathcal G)$. It follows from Lemma
  \ref{53} that $\phi(p_n) < \infty$ for all $n$. Since $\phi \neq 0$
  we can assume, without loss of generality, that $\phi(p_n) > 0$ for
  all $n$. Since $\phi(f) < \infty$ for every non-negative function in
  $C_c\left(\mathcal G^{(0)}\right)$ it follows that $C_c(\mathcal
  G) \subseteq \mathcal M_{\phi}$ and from the Riesz
  representation theorem that there is a unique regular Borel measure
  $\mu$ on $\mathcal G^{(0)}$ such that
$$
\phi(f) = \int_{\mathcal G^{(0)}} f(x) \ d\mu(x)
$$
for all $f  \in C_c(\mathcal G^{(0)})$. Let $U_n$ be the compact and
open support of $p_n$, and set 
$$
\mathcal G^n = \mathcal G|_{U_n} = \left\{ \xi \in \mathcal G: \
  r(\xi),s(\xi) \in U_n\right\}
$$ 
and $c_n
= c|_{\mathcal G^{n}}$. Note that $\phi(p_n)^{-1}\phi$ restricts to a
$\beta$-KMS state on $p_nC^*(\mathcal G)p_n = C^*\left(\mathcal
  G^n\right)$. It follows from Neshveyev's
theorem \cite{N} that there is a probability measure $\mu_{n}$ on $U_{n}$, and a $\mu_{n}$-measurable field
$\left\{\varphi^n_x\right\}_{ x \in U_n}$ of states such that
\begin{enumerate}
\item[an)] $\mu_{n}$ is quasi-invariant on $\mathcal{G}^{n}$ with cocycle $e^{-\beta c_n}$.
\item[bn)] $\varphi^n_x(u_g) = \varphi^n_{r(h)}\left(u_{hgh^{-1}}\right)$
    for $\mu_{n}$-almost every $x\in U_n$ and all $g\in
    \mathcal G^x_x, \ h \in (\mathcal G^n)_x$,
\item[cn)] $\varphi^n_x(u_g)  = 0$ for $\mu_{n}$-almost every $x \in U_n$ and all $g \in \mathcal G^x_x \backslash c_{n}^{-1}(0)$,
\end{enumerate}
and
$$
\phi(p_n)^{-1} \phi(f) = \int_{U_n}  \sum_{g \in \mathcal G^x_x}
f(g)\varphi^n_x(u_g) \ d\mu_{n}(x) 
$$
when $f \in C_c(\mathcal G^n)$. For every $f \in C_{c}(U_{n})$ we get that:
\begin{align*}
&\phi(p_n)^{-1}\int_{U_{n}} f(x) \ d\mu(x)=\phi(p_n)^{-1}\int_{\mathcal G^{(0)}} f(x) \ d\mu(x)=\phi(p_n)^{-1} \phi(f) \\
&=\int_{U_n}  \sum_{g \in \mathcal G^x_x}
f(g)\varphi^n_x(u_g) \ d\mu_{n}(x) =\int_{U_n}  f(x) \ d\mu_{n}(x)
\end{align*}
so $\mu|_{U_{n}}=\phi(p_{n}) \mu_{n}$. Notice that since $\phi(p_{n})>0$, being a $\mu$ nullset in $U_{n}$ is the same as being a $\mu_{n}$ nullset. For a Borel set $V \subseteq U_{n} \subseteq U_{n+1}$ we have that:
$$
\phi(p_{n+1})\mu_{n+1}(V) = \mu(V)= \phi(p_{n})\mu_{n}(V)
$$
So $\mu_{n} = \phi(p_{n+1})/\phi(p_{n}) \mu_{n+1}|_{U_{n}}$. For every $f \in C_{c}(\mathcal{G}^{n})$ we get that:
\begin{align*}
&\int_{U_{n}}\sum_{g \in \mathcal G^x_x} f(g)\varphi^n_x(u_g) \ d\mu_{n}(x) = \phi(p_n)^{-1} \phi(f)=\frac{\phi(p_{n+1})}{\phi(p_n)}\phi(p_{n+1})^{-1} \phi(f) \\
&=\frac{\phi(p_{n+1})}{\phi(p_n)} \int_{U_{n+1}}\sum_{g \in \mathcal G^x_x} f(g)\varphi^{n+1}_x(u_g) \ d\mu_{n+1}(x) = \int_{U_{n}}\sum_{g \in \mathcal G^x_x} f(g)\varphi^{n+1}_x(u_g) \ d\mu_{n}(x)
\end{align*}
Since $\mu_{n}$ by choice satisfy an), and since it's easily seen that
$\{\varphi^{n+1}_{x}\}_{x \in U_{n}}$ satisfy bn) and cn), the
uniqueness statement in Neshveyev's theorem gives that
$\varphi^{n}_{x}=\varphi^{n+1}_{x}$ for a.e. $x \in U_{n}$. Hence for
a.e. $x \in \mathcal{G}^{(0)}$ we can define a state on
$C^*\left(\mathcal G^x_x\right)$ by:
$$
\varphi_{x}(d) = \lim_{n \to\infty} \varphi^n_x(d).
$$
For every $f \in C_{c}(\mathcal{G})$ there is a $ N \in \mathbb{N}$ such that $f \in C_{c}(\mathcal{G}^{N})$, and hence:
$$
\phi(f)=\phi(P_{N})\int_{U_{N}}\sum_{g \in \mathcal G^x_x}
f(g)\varphi^N_x(u_g) \ d\mu_{N}(x) =\int_{\mathcal{G}^{(0)}}\sum_{g
  \in \mathcal G^x_x} f(g)\varphi_x(u_g) \ d\mu(x) .
$$
The properties a)-c) follow from an)-cn), and measurability of $x \mapsto \sum_{g \in 
  \mathcal G^x_x} f(g) \varphi_x(u_g) $ follows from measurability of $x \mapsto \sum_{g \in 
  \mathcal G^x_x} f(g) \varphi^{n}_x(u_g) $.

For the converse, assume we are given a pair $\left(\mu,
  \left\{\varphi_x\right\}_{x \in \mathcal G^{(0)}}\right)$ for which
a),b) and c) hold. As shown
by Neshveyev in the proof of Theorem 1.1 in \cite{N} every $x$ gives rise to a state $\psi_x$ on $C^*(\mathcal G)$
such that
$$
\psi_x (f) = \sum_{g \in \mathcal G_x^x} f(g)\varphi_x(u_g)
$$
when $f \in C_c(\mathcal G)$. Note that
$x \to \sum_{g \in \mathcal G_x^x} f(g)\varphi_x(u_{g})$ is
$\mu$-measurable by assumption, and then $x \to \psi_x(a)$ is also for each $a
\in C^*(\mathcal G)$. For $a \geq 0$ we can therefore define
$$
\phi(a) = \int_{\mathcal G^{(0)}} \psi_x(a) \
d\mu(x) .
$$
$\phi$ is a lower semi-continuous weight by Fatous lemma and by regularity
$$
\phi(p_nap_n) =  \int_{U_n} \psi_x(a) \
d\mu(x)
\leq \|a\|  \mu(U_{n}) < \infty
$$
for all $n$, so it is also densely defined. Note that $C_c(\mathcal G)
\subseteq \mathcal M_{\phi}$ and that (\ref{60}) holds by
construction. Since the pair
$\left(\phi(p_{n})^{-1}\mu,   \left\{\varphi_x\right\}_{x \in
    U_{n}}\right)$ represents $\phi(p_{n})^{-1} \phi$ in the sense of
Theorem 1.1 in \cite{N} it follows from Theorem 1.3 in \cite{N} that $\phi$ is a bounded $\beta$-KMS weight on
$p_nC^*(\mathcal G)p_n$. Since
$$
\psi_x(p_n a p_n) = \begin{cases}  \psi_x(a), & \ x \in U_n \\ 0, & \ x \notin U_n
\end{cases},
$$
we find that $\lim_{n \to \infty} \phi(p_nap_n) = \lim_{n \to \infty}
\int_{U_n}  \psi_x(a) \
d\mu(x) = \phi(a)$ for all $a \geq 0$ in $C^*(\mathcal G)$. Now note
that for every $a$ in the domain of $\sigma^c_{\frac{-i\beta}{2}}$,
$$
\phi(p_na^*ap_n) = \phi\left(
  \sigma^c_{\frac{-i\beta}{2}}(ap_n)\sigma^c_{\frac{-i\beta}{2}}(ap_n)^*\right) =  \phi\left(
  \sigma^c_{\frac{-i\beta}{2}}(a)p_n\sigma^c_{\frac{-i\beta}{2}}(a)^*\right)
$$
since $\phi$ is a bounded $\beta$-KMS weight on $p_nC^*(\mathcal
G)p_n$. Since
$$
\lim_{n \to\infty}  \phi\left(
  \sigma^c_{\frac{-i\beta}{2}}(a)p_n\sigma^c_{\frac{-i\beta}{2}}(a)^*\right) =  \phi\left(
  \sigma^c_{\frac{-i\beta}{2}}(a)\sigma^c_{\frac{-i\beta}{2}}(a)^*\right)
$$
by the lower semi-continuity of $\phi$, we conclude that $\phi(a^*a) = \phi\left(
  \sigma^c_{\frac{-i\beta}{2}}(a)\sigma^c_{\frac{-i\beta}{2}}(a)^*\right)$, showing that $\phi$ is indeed a $\beta$-KMS weight for $\sigma^c$.

If $\left(\mu,
  \left\{\varphi_x\right\}_{x \in \mathcal G^{(0)}}\right)$ and  $\left(\mu',
  \left\{\varphi'_x\right\}_{x \in \mathcal G^{(0)}}\right)$ represent
the same $\beta$-KMS weight it follows from the uniqueness part of the
Riesz representation theorem that $\mu = \mu'$. By using (\ref{60}) we
find that 
\begin{equation}\label{61}
\int_{\mathcal G^{(0)}}  k(x)\sum_{g \in \mathcal G^x_x}
f(g)\varphi_x(u_g) \ d\mu(x) = \int_{\mathcal G^{(0)}}  k(x)\sum_{g \in \mathcal G^x_x}
f(g)\varphi'_x(u_g) \ d\mu(x) 
\end{equation}
when $f \in C_c(\mathcal G)$ and $k \in C_c\left(\mathcal
  G^{(0)}\right)$. It follows from this that
$$
\sum_{g \in \mathcal G^x_x}
f(g)\varphi_x(u_g) = \sum_{g \in \mathcal G^x_x}
f(g)\varphi'_x(u_g) 
$$
for $\mu$-almost all $x \in \mathcal G^{(0)}$ and all $f \in
C_c\left(\mathcal G\right)$. Thanks to the separability of
$C^*(\mathcal G)$ we conclude that $\varphi_x = \varphi'_x$ for
$\mu$-almost all $x$.

\end{proof}

\begin{cor}\label{65} Let $\mathcal G$ be a locally compact second countable
  Hausdorff \'etale groupoid and let $c : \mathcal G \to \mathbb R$ be
  a continuous homomorphism. Assume that the unit space $\mathcal G^{(0)}$ of $\mathcal G$ is totally
  disconnected and that the isotropy groups $\mathcal G^x_x, x \in
  \mathcal G^{(0)}$, are all amenable.

There is a bijective correspondence
  between the $\beta$-KMS weights for $\sigma^c$ on
  $C^*_r(\mathcal G)$ and the pairs $\left(\mu,
  \left\{\varphi_x\right\}_{x \in \mathcal G^{(0)}}\right)$, where
$\mu$ is a regular Borel measure on $\mathcal G^{(0)}$ and
$\left\{\varphi_x\right\}_{x \in \mathcal G^{(0)}}$ is a
$\mu$-measurable field of states $\varphi_x$ on
$C^*_r(\mathcal G^x_x)$ such that
\begin{enumerate}
\item[a)] $\mu$ is quasi-invariant with cocycle $e^{-\beta c}$,
\item[b)] $\varphi_x(u_g) = \varphi_{r(h)}\left(u_{hgh^{-1}}\right)$
    for $\mu$-almost every $x\in \mathcal G^{(0)}$ and all $g\in
    \mathcal G^x_x, \ h \in \mathcal G_x$, and
\item[c)] $\varphi_x(u_g)  = 0$ for $\mu$-almost every $x \in \mathcal
  G^{(0)}$ and all $g \in \mathcal G^x_x \backslash c^{-1}(0)$.
\end{enumerate}
The $\beta$-KMS weight $\phi$ corresponding to the pair $\left(\mu,
  \left\{\varphi_x\right\}_{x \in \mathcal G^{(0)}}\right)$ has the
properties that $C_c\left(\mathcal G\right) \subseteq \mathcal
M_{\phi}$ and
\begin{equation}\label{66}
\phi(f) = \int_{\mathcal G^{(0)}}  \sum_{g \in \mathcal G^x_x}
f(g)\varphi_x(u_g) \ d\mu(x) 
\end{equation}
when $f \in C_c(\mathcal G)$.
\end{cor}
\begin{proof}  It suffices to show that the assumption on the isotropy
  groups implies that every $\beta$-KMS weight $\phi$ on $C^*(\mathcal G)$
    factorises through $C^*_r(\mathcal G)$. To this end note that it
    follows from
    Lemma 2.1 in \cite{Th2} that for each $n \in \mathbb N$ there is a
    bounded $\beta$-KMS weight
    $\tilde{\phi}_n$ on $p_n C^*_r(\mathcal G)p_n$ such that
    $\tilde{\phi}_n(p_n\pi(a)p_n) = \phi(p_nap_n)$ for all $a \in
    C^*(\mathcal G)$ where $\pi : C^*(\mathcal G) \to C^*_r(\mathcal
    G)$ is the canonical surjection. Then $\tilde{\phi}_n(p_nbp_n) \leq
    \tilde{\phi}_{n+1}(p_{n+1}bp_{n+1})$ for all $b \geq 0$ in
    $C^*_r(\mathcal G)$ and we can define a lower semi-continuous
    weight $\tilde{\phi}$ on $C^*_r(\mathcal G)$ such that
    $\tilde{\phi}(b) = \lim_{n \to\infty}
    \tilde{\phi}_n(p_nbp_n)$. It follows that $\tilde{\phi} \circ \pi =
    \phi$.

\end{proof}

It is an interesting problem if Corollary \ref{65} remains true
without the amenability assumption on the isotropy groups. For the
proof of our main result the following suffices.

\begin{cor}\label{75}  Let $\mathcal G$ be a locally compact second countable
  Hausdorff \'etale groupoid and let $c : \mathcal G \to \mathbb R$ be
  a continuous homomorphism. Assume that the unit space $\mathcal G^{(0)}$ of $\mathcal G$ is totally
  disconnected. If there is a $\beta$-KMS weight for $\sigma^c$ on $C^*_r(\mathcal G)$
there is also one which is diagonal.
\end{cor}
\begin{proof} Let $\phi$ be a $\beta$-KMS weight for $\sigma^c$ on
  $C^*_r(\mathcal G)$ and let $\pi : C^*(\mathcal G) \to
  C^*_r(\mathcal G)$ be the canonical surjection. Then $\phi \circ
  \pi$ is a $\beta$-KMS weight for $\sigma^c$ on
  $C^*(\mathcal G)$ and we can consider the corresponding regular
  Borel measure $\mu$. Since $\mu$ is quasi-invariant with cocycle
  $e^{-\beta c}$ it follows from Proposition 2.1 in \cite{Th3} that
  $\mu$ defines a diagonal $\beta$-KMS weight by the formula (\ref{f1}).
\end{proof}

\section{Conditions on a KMS weight that imply diagonality of the action}

\subsection{When KMS weights factor through
  the conditional expectation onto an abelian subalgebra}

A weight $\omega$ is \emph{faithful} when $a \geq 0,\ \omega(a) = 0
\Rightarrow a =0$.

\begin{prop}\label{4} Let $A$ be a $C^*$-algebra and $\gamma$ a
  continuous one-parameter group of automorphisms on $A$. Let $D \subseteq A$ be an
  abelian $C^*$-subalgebra and $P : A \to D$ a conditional
  expectation.

Assume that $\omega$ is a faithful $\beta$-KMS weight for $\gamma$, $\beta \neq
0$, such that $\omega \circ P = \omega$. It follows that $\gamma_t(d)
= d$ for all $t \in \mathbb R$ and all $d \in D$.
\end{prop}
\begin{proof} Let $f \in D, \ f \geq 0$. Since $\omega$ is densely
  defined there is a sequence $\{a_n\}$ of positive elements in $A$
  such that $\lim_{n \to \infty} a_n = f$ and $\omega(a_n) < \infty$
  for all $n$. Then $\lim_{n \to \infty} P(a_n) = f$ and $\omega
  (P(a_n)) = \omega(a_n) < \infty$. It suffices therefore to consider
  $f \in D, \ f \geq 0$ such that $\omega(f) < \infty$ and show that
  $\gamma_t(f) = f$ for all $t \in \mathbb R$. 

We find that
\begin{equation}\label{5}
\omega(af ) = \omega(P(a)f) = \omega(fP(a)) = \omega(fa)
\end{equation}
for all $a\in \mathcal M_{\omega}$. Since $\gamma_t(\mathcal M_{\omega}) = \mathcal M_{\omega}$
and $\omega \circ \gamma_t = \omega$ for all $t$, it follows from
(\ref{5}) that
\begin{equation}\label{5'}
\omega(a\gamma_t(f) ) =  \omega(\gamma_t(f)a)
\end{equation}
for all $a\in \mathcal M_{\omega}$ and all $t \in \mathbb R$. For $k \in \mathbb N$, $a \in A$, set
$$
Q_k(a) = \sqrt{\frac{k}{\pi}} \int_{\mathbb R} e^{-kt^2} \gamma_t(a) \
dt .
$$
Note that $Q_k(a)$ is analytic for $\gamma$ and that $\lim_{k \to
  \infty} Q_k(a) = a$.

\begin{obs}\label{6} Assume that $a \in \mathcal M_{\omega}$. It follows that 
$$
 \sqrt{\frac{k}{\pi}} \int_{\mathbb R} e^{-k(t + is)^2} \gamma_t(a) \
dt \in \mathcal M_{\omega}
$$
for all $s\in \mathbb R$.
\end{obs}
Indeed,
$$
 \sqrt{\frac{k}{\pi}} \int_{\mathbb R} e^{-k(t + is)^2} \gamma_t(a) \
dt = e^{ks^2} \sqrt{\frac{k}{\pi}} \int_{\mathbb R} e^{-kt^2} (\cos
(2kst) - i\sin (2kst)) \gamma_t(a) \ dt ,
$$
so it suffices to show that
$$
x = \sqrt{\frac{k}{\pi}} \int_{\mathbb R} g(t) e^{-kt^2} \gamma_t(a) \
dt \in \mathcal M_{\omega}
$$
when $a \geq 0, \ \omega(a) < \infty$ and $g \geq 0$ is continuous and bounded. Since
$$
\sqrt{\frac{k}{\pi}} \int_{\mathbb R} e^{-kt^2} \ dt = 1,
$$
it follows that $x$ is the norm limit of a sequence $\{x_n\}$ of elements of the
form
$$
x_n = \sum_{i=1}^N \alpha_i g(t_i) \gamma_{t_i}(a) 
$$
where $\alpha_i \geq 0$ and $\sum_{i=1}^N \alpha_i = 1$. Note that 
$$
\omega(x_n) = \sum_{i=1}^N \alpha_i g(t_i) \omega(\gamma_{t_i}(a)) =
\sum_{i=1}^N \alpha_i g(t_i) \omega(a) \leq \omega(a) \sup_{t \in
  \mathbb R} g(t) .
$$
Since $\omega$ is lower semi-continuous, it follows that
$\omega(x)\leq \omega(a) \sup_{t \in
  \mathbb R} g(t) < \infty$, completing the proof of Observation
\ref{6}.

Since
$$
\gamma_{is}(Q_k(a)) =  \sqrt{\frac{k}{\pi}} \int_{\mathbb R}
e^{-k(t-is)^2} \gamma_t(a) \ dt,
$$
it follows from Observation \ref{6} that 
\begin{equation}\label{8}
\gamma_{is}\left(Q_k(a)\right) \in \mathcal M_{\omega}
\end{equation}
when $a \in \mathcal M_{\omega}$ and for all $s \in \mathbb R$. Let $a,b \in \mathcal M_{\omega}, \ t \in \mathbb R$. Assume that $a$ and $b$
are analytic for $\gamma$, and that $\gamma_{i\beta}(a),
\gamma_{i\beta}(b) \in \mathcal M_{\omega}$. Then
\begin{equation}\label{7}
\begin{split}
&\omega \left(a\gamma_t(f)b\right) = \omega\left(
  \gamma_t(f)b\gamma_{i\beta}(a)\right) \  \ \ \ \ \ \ \ \ \ \ \ \ \ \ \text{(using Proposition 1.12 (3) in \cite{KV})}\\
& = \omega\left( b\gamma_{i\beta}(a)\gamma_t(f)\right)  \ \ \ \ \ \
\  \ \ \ \  \ \ \ \ \ \ \ \ \ \ \ \ \ \ \ \ \ \text{(using (\ref{5'}))} \\
& = \omega\left( \gamma_{i\beta}(a)\gamma_t(f)
  \gamma_{i\beta}(b)\right) \ \ \  \ \  \ \ \ \ \ \ \ \ \ \ \ \ \ \ \text{(using Proposition 1.12 (3) in
  \cite{KV})} .
\end{split}
\end{equation}

Note that when $a,c \in \mathcal M_{\omega}$, the Cauchy-Schwarz inequality
yields the estimate
$$
\left|\omega(abc)\right|  \leq \|b\| \omega(aa^*)^{\frac{1}{2}}
\omega(c^*c)^{\frac{1}{2}} 
$$
for all $b \in A$. In particular,
$x \mapsto \omega \left( axb\right)$
and
$x \mapsto \omega\left( \gamma_{i\beta}(a)x
  \gamma_{i\beta}(b)\right)$
are both bounded functionals on $A$ and we conclude therefore from
(\ref{7}) that
\begin{equation*}\label{9}
\omega \left( aQ_k(f)b\right) = \omega\left( \gamma_{i\beta}(a)Q_k(f)
  \gamma_{i\beta}(b)\right)
\end{equation*}
for all $k \in \mathbb N$. Since
\begin{equation*}
\begin{split}
&\omega\left( \gamma_{i\beta}(a)Q_k(f)
  \gamma_{i\beta}(b)\right) = \omega \circ \gamma_{i\beta}\left( a\gamma_{-i\beta}(Q_k(f))b\right)  \\
&= \omega \left( a\gamma_{-i\beta}(Q_k(f))b\right) ,
\end{split}
\end{equation*}
it follows that
\begin{equation}\label{a11}
\omega \left( a\left[Q_k(f) - \gamma_{-i\beta}(Q_k(f))\right] b\right)
=  0 .
\end{equation}
Note now that it follows from (\ref{8}) that $Q_k(f) -
\gamma_{-i\beta}(Q_k(f)) \in \mathcal M_{\omega}$.  Since $\mathcal M_{\omega}$ is a
$*$-algebra we conclude that
$$
\left(Q_k(f) -
\gamma_{-i\beta}(Q_k(f))\right)^* \ \text{and} \ 
 \left(Q_k(f) -
\gamma_{-i\beta}(Q_k(f))\right)^*  \left(Q_k(f) -
\gamma_{-i\beta}(Q_k(f))\right) 
$$
are also elements of $\mathcal M_{\omega}$. Similarly,
\begin{equation*}
\gamma_{i\beta} \left( \left(Q_k(f) -
\gamma_{-i\beta}(Q_k(f))\right)^*\right) = \gamma_{i\beta}(Q_k(f)) -
\gamma_{2i\beta} (Q_k(f)) \in \mathcal M_{\omega}
\end{equation*}
and
\begin{equation*}
\begin{split}
&\gamma_{i\beta} \left( \left(Q_k(f) -
\gamma_{-i\beta}(Q_k(f))\right)^*  \left(Q_k(f) -
\gamma_{-i\beta}(Q_k(f))\right)\right) \\
&= \gamma_{i\beta}\left(  \left(Q_k(f) -
\gamma_{-i\beta}(Q_k(f))\right)^* \right) \gamma_{i\beta}\left(Q_k(f) -
\gamma_{-i\beta}(Q_k(f)) \right)  \ \in \ \mathcal M_{\omega} .
\end{split}
\end{equation*}
We can then choose 
$$
a = \left(Q_k(f) -
\gamma_{-i\beta}(Q_k(f))\right)^*
$$
and
$$
b =  \left(Q_k(f) -
\gamma_{-i\beta}(Q_k(f))\right)^*  \left(Q_k(f) -
\gamma_{-i\beta}(Q_k(f))\right)  
$$
in (\ref{a11}) to deduce that
$\omega(b^2) = 0$. Since $\omega$ is faithful by assumption it follows that $b = 0$, i.e.
$$
Q_k(f) = \gamma_{-i\beta}(Q_k(f)) .
$$ 
Since $Q_k(f)$ is analytic for $\gamma$ it follows that
\begin{equation}\label{13}
\gamma_z(Q_k(f)) = \gamma_{z-i \beta}(Q_k(f))
\end{equation}
for all $ z \in \mathbb C$. Set
$$
D = \sup \left\{\left\| \gamma_{is} (Q_k(f))\right\| : \ s \in \left[
    -|\beta|,|\beta| \right] \right\}  .
$$
Using (\ref{13}) it follows that
$\left\| \gamma_{t+is}(Q_k(f))\right\| = \left\|
  \gamma_{is}(Q_k(f))\right\| \leq D$
for all $t,s \in \mathbb R$. Then Liouvilles theorem implies
that 
$$
z \mapsto \gamma_z(Q_k(f))
$$
is constant, i.e. $\gamma_t(Q_k(f)) = Q_k(f)$ for all $t \in \mathbb
R$. This conclusion holds for all $k \in \mathbb N$ and since $\lim_{k
  \to \infty} Q_k(f) = f$, this completes the proof.
\end{proof}

\begin{cor}\label{77} Let $A$ be a simple $C^*$-algebra and $\gamma$ a
  continuous one-parameter group of automorphisms on $A$. Let $D \subseteq A$ be an
  abelian $C^*$-subalgebra and $P : A \to D$ a conditional
  expectation.

Assume that $\omega$ is a $\beta$-KMS weight for $\gamma$, $\beta \neq
0$, such that $\omega \circ P = \omega$. It follows that $\gamma_t(d)
= d$ for all $t \in \mathbb R$ and all $d \in D$.
\end{cor}
\begin{proof} It suffices to show that $\omega$ is faithful. So assume
  that $b = b^* \in A$ and that $\omega(b^2) = 0$. For a $c \in \mathcal{M}_{\omega}$ we know from the above proof that $Q_{k}(c), \gamma_{i\beta} (Q_{k}(c)^{*}) \in \mathcal{M}_{\omega}$, hence by an application of the Cauchy-Schwarz inequality
$$
| \omega(Q_{k}(c)^{*}b^{2} Q_{k}(c)) |^{2} = | \omega(b^{2}
Q_{k}(c)\gamma_{i\beta} (Q_{k}(c)^{*})) |^{2} \leq 0 .
$$
Lower semi-continuity now implies that $\omega(c^{*}b^{2} c)=0$ and by using Cauchy-Schwarz again we deduce that 
\begin{equation}\label{11}
\omega\left( \Span \mathcal M_{\omega} b^2 \mathcal M_{\omega}\right) = \{0\}.
\end{equation}
Since $\mathcal M_{\omega}$ is dense in $A$ the closure of $ \Span \mathcal M_{\omega}
b^2 \mathcal M_{\omega}$ is a (closed twosided) ideal in $A$. If $b \neq 0$
this ideal must be all of $A$ because we assume that $A$ is
simple. But then we reach a contradiction the following way: Let $ a
\geq 0$. Choose a sequence $\{x_n\} \subseteq  \Span \mathcal M_{\omega} b^2
\mathcal M_{\omega}$ such that $\lim_{n \to \infty} x_n = \sqrt{a}$. Since
$x_nx_n^* \in  \Span \mathcal M_{\omega} b^2 \mathcal M_{\omega}$ and $\lim_{n \to
  \infty} x_nx_n^* = a$, it follows from (\ref{11}) and the lower
semi-continuity of $\omega$ that $\omega(a) = 0$. This is a
contradiction because $\omega \neq 0$. Hence $b = 0$.

\end{proof}

\subsection{One-parameter groups trivial on the diagonal}

\begin{prop}\label{15} Let $\mathcal G$ be a locally compact Hausdorff
  \'etale groupoid and $\alpha = (\alpha_t)_{t \in \mathbb R}$ a
  continuous one-parameter group of automorphisms on $C^*_r(\mathcal
  G)$ such that
$$
\alpha_t(f) = f
$$
for all $f \in C_0\left(\mathcal G^{(0)}\right)$ and all $t \in
\mathbb R$. Assume that the elements of $\mathcal G^{(0)}$ with
trivial isotropy group in $\mathcal G$ are dense in $\mathcal G^{(0)}$. There is a continuous homomorphism $c : \mathcal G \to \mathbb R$ such
that
$$
\alpha_t(g)(\xi) = e^{it c(\xi)} g(\xi)
$$
for all $t\in \mathbb R$, all $g \in C_c(\mathcal G)$ and all $\xi \in
\mathcal G$. 

\end{prop}

\begin{proof} We shall use the continuous linear embedding $j :
  C^*_r\left(\mathcal G\right) \to C_0(\mathcal G)$ introduced by
  Renault in Proposition 4.2 in \cite{Re1}.

\begin{obs}\label{16} Let $f \in C_c(\mathcal G)$ be supported in an open subset $U
  \subseteq \mathcal G$ such that $r: U \to \mathcal G^{(0)}$ is injective. Assume that $f(\xi) = 0$ for some $\xi \in
  U$. It follows that $j(\alpha_t(f))(\xi) = 0$ for all $t\in \mathbb R$. 
\end{obs} 
To prove this, let $\epsilon > 0$. There is an open bisection $W$ of $\xi$ such
that $W \subseteq U$ and $|f(\mu)| \leq \epsilon$ for all $\mu \in
W$. Let $\varphi \in C_c\left(\mathcal G^{(0)}\right)$ be such that $0
\leq \varphi \leq 1$, $\supp \varphi \subseteq r(W)$ and $\varphi(r(\xi))
= 1$. By use of Proposition 4.2 in \cite{Re1} we find that 
\begin{equation}\label{17}
j(\alpha_t(f))(\xi) = \varphi(r(\xi))j(\alpha_t(f))(\xi) = j(\varphi
\alpha_t(f))(\xi) = j(\alpha_t(\varphi f))(\xi).
\end{equation}
Note that 
$\supp (\varphi f) \subseteq W$ and that $\left\|\varphi f\right\|_{\infty} \leq
\epsilon$. It follows that
$$
\left\|j(\alpha_t(\varphi f))\right\|_{\infty}  \leq
\left\|\alpha_t(\varphi f)\right\| = \left\|\varphi
  f\right\| = \left\|\varphi f\right\|_{\infty} \leq\epsilon ,
$$
where the last identity follows from Lemma 2.4 in \cite{Th1}.
In particular, $\left| j(\alpha_t(\varphi f))(\xi)\right| \leq \epsilon$,
and then (\ref{17}) shows that $\left|j(\alpha_t(f))(\xi)\right| \leq
\epsilon$. This proves Observation \ref{16}.

In the same way we obtain the following
\begin{obs}\label{16'} Let $f \in C_c(\mathcal G)$ be supported in an open subset $U
  \subseteq \mathcal G$ such that $s: U \to \mathcal G^{(0)}$ is injective. Assume that $f(\xi) = 0$ for some $\xi \in
  U$. It follows that $j(\alpha_t(f))(\xi) = 0$ for all $t\in \mathbb R$. 
\end{obs}

\begin{obs}\label{18} Let $\xi \in \mathcal G$, and let $h, h' \in
  C_c(\mathcal G)$ be supported in (not necessarily the same) open
  bisections in $\mathcal G$. Assume that $h(\xi) = h'(\xi) = 1$. Then
\begin{equation}\label{ab19}
j(\alpha_t(h))(\xi) = j\left(\alpha_t(h')\right)(\xi)
\end{equation}
for all $t \in \mathbb R$.
\end{obs}
To show this, let $h\cdot h'$ be the pointwise product of $h$ and
$h'$. It follows from Observation \ref{16} that
$$
j(\alpha_t(h\cdot h' - h'))(\xi ) = j(\alpha_t(h\cdot h' - h))(\xi ) =
0,
$$
which yields (\ref{ab19}): $j(\alpha_t(h))(\xi) = j(\alpha_t(h\cdot
h'))(\xi) = j(\alpha_t(h'))(\xi)$.

It follows from Observation \ref{18} that we can define a map $G_t :
\mathcal G \to \mathbb C$ such that
$$
G_t(\xi) = j(\alpha_t(h))(\xi),
$$
where $h$ is any element of $C_c(\mathcal G)$ which is supported in an
open bisection and takes the value $1$ at $\xi$. Note that $G_t$ is
continuous by construction.

\begin{obs}\label{20} For every $f \in C_c(\mathcal G)$ and every $\xi
  \in \mathcal G$,
\begin{equation}\label{21}
j(\alpha_t(f))(\xi) = G_t(\xi) f(\xi).
\end{equation}
\end{obs}

To show this, we may assume that there are open bisections $U
\subseteq V$ such that $\supp f \subseteq U$ and $\overline{U}
\subseteq V$. Assume first that $\xi \notin \overline{U}$. We must
show that $j(\alpha_t(f))(\xi) = 0$ in this case. By continuity and
the assumption on $\mathcal G$ we may assume that $s(\xi)$ has
trivial isotropy. If $\mu \in U$ and $r(\mu) = r(\xi), \ s(\mu) =
s(\xi)$, we see that
$$
r(\mu^{-1} \xi) = s(\mu) = s(\xi) \ \text{and} \
s(\mu^{-1}\xi) = s(\xi)
$$  
which is impossible since $\xi  \neq \mu$. It follows that we can
write $f$ as a finite sum
$$
f = \sum_i f_i
$$
such that each $f_i \in C_c(U)$ is supported in an open set $W_i
\subseteq U$ such
that either $s(\xi) \notin s(\overline{W_i})$ or $r(\xi) \notin
r(\overline{W_i})$. It follows that $j(\alpha_t(f_i))(\xi) = 0$; in
the first case thanks to Observation
\ref{16'}, in the second thanks to
Observation \ref{16}. Hence
$$
j(\alpha_t(f))(\xi) = \sum_ij(\alpha_t(f_i))(\xi) = 0 ,
$$
as desired. Assume then that $\xi \in \overline{U} \subseteq
V$. Choose $\epsilon > 0$ such that
$f(\xi) + \epsilon \neq 0$ and a
function $\varphi \in C_c(V)$ such that $\varphi(\xi) = 1$. Then 
$$
j(\alpha_t(f + \epsilon \varphi))(\xi) = 
j\left(\alpha_t\left(\frac{f + \epsilon \varphi}{f(\xi) + \epsilon}\right)\right)(\xi)(f(\xi) + \epsilon)
=G_t(\xi)(f(\xi) + \epsilon) .
$$
Letting  $\epsilon \to 0$ we obtain (\ref{21}).

Note that it follows from Observation \ref{20} that
$\alpha_t\left(C_c(\mathcal G)\right) \subseteq C_c(\mathcal G)$, and
\begin{equation*}\label{20'}
\alpha_t(f)(\xi) = G_t(\xi)f(\xi)
\end{equation*}
for all $f \in C_c(\mathcal G)$ and all $\xi \in \mathcal G$. Since
$\left\|f\right\| = \left\|\alpha_t(f)\right\|$ this implies that
$\left|G_t(\xi) \right| = 1$. Furthermore, if $h \in C_c(\mathcal G)$
is supported in a bisection and $h(\xi) = 1$, we find that
\begin{equation*}\label{24}
\begin{split}
&G_{t+s}(\xi) = \alpha_t\left(\alpha_s(h)\right)(\xi) \\
& =
\alpha_s(h)(\xi)
\alpha_t\left(\frac{\alpha_s(h)}{\alpha_s(h)(\xi)}\right)(\xi) =
\alpha_s(h)(\xi) G_t(\xi) \\
& = G_s(\xi)G_t(\xi) .
\end{split}
\end{equation*}
Since $t \mapsto G_t(\xi)$ is continuous, this implies that there is a unique
real-valued function $c : \mathcal G \to \mathbb R$ such that
\begin{equation}\label{25}
G_t(\xi) = e^{it c(\xi)} .
\end{equation}
To show that $c$ is a homomorphism, let
$\gamma_1,\gamma_2 \in \mathcal G$ such that $s(\gamma_1) =
r(\gamma_2)$. Set $\gamma = \gamma_1\gamma_2$. Let $U$ be an open
bisection containing $\gamma$ and $U_i$ an open
bisection containing $\gamma_i, \ i =1,2$, such that $\mu_1\mu_2 \in
U$ when $(\mu_1,\mu_2) \in \mathcal G^{(2)} \cap (U_1 \times
U_2)$. Choose $h_i \in C_c(U_i)$ such that $h_i(\gamma_i) = 1$. Then
$h_1h_2(\gamma) = 1$ and
\begin{equation}\label{26}
G_t(\gamma) = j(\alpha_t(h_1h_2))(\gamma) =
\alpha_t(h_1)\alpha_t(h_2)(\gamma) =
\alpha_t(h_1)(\gamma_1)\alpha_t(h_2)(\gamma_2) = G_t(\gamma_1)G_t(\gamma_2).
\end{equation}
Hence $G_t$ is a homomorphism as asserted. Combining (\ref{25}) and
(\ref{26}) and taking derivatives with respect to $t$, it follows
that $c$ is a homomorphism, i.e.
$$
c(\gamma_1 \gamma_2) = c(\gamma_1) + c(\gamma_2)
$$
when $s(\gamma_1) = r(\gamma_2)$. 

Finally, to show that $c$ is continuous, let $\xi \in \mathcal
G$ and $\epsilon > 0$ be given. Choose open bisections $
U \subseteq V$ such that $\xi \in U \subseteq \overline{U}
\subseteq V$ and $h \in C_c(V)$ a function such that $h =1$ on
$\overline{U}$. Then
$$
G_t(\gamma) = \alpha_t(h)(\gamma)
$$
for all $t \in \mathbb R$ and all $\gamma \in U$. Let $K \subseteq \mathbb R$
be a compact set. There are finitely many points $t_i \in K,i = 1,2,
\cdots, N$, such that for every $t \in K$ there is an $i$ such that
$$
\left\|\alpha_t(h) -\alpha_{t_i}(h)\right\|_{\infty}  =
\left\|\alpha_t(h) -\alpha_{t_i}(h)\right\| \leq \epsilon .
$$
By continuity of $\alpha_{t_i}(h)$ there is an open neighborhood $W
\subseteq U$
of $\xi$ such that 
$$
\left|\alpha_{t_i}(h)(\gamma) - \alpha_{t_i}(h)(\xi)\right| \leq
\epsilon
$$
for all $\gamma \in W$ and $i=1,2, \dots, N$. It follows that $\left|G_t(\gamma)
  -G_t(\xi)\right| \leq 3 \epsilon$ for all $t \in K$ and all $\gamma
\in W$. By Pontryagin duality this implies that $c$ is continuous.

\end{proof}

\begin{thm}\label{28} Let $\mathcal G$ be a locally compact Hausdorff
  \'etale groupoid such that for at least one element $x \in \mathcal
  G^{(0)}$ the isotropy $\mathcal G_x^x$ is trivial, i.e. $\mathcal
  G_x^x = \{x\}$, and that $\mathcal G$ is minimal in the sense that
$s(r^{-1}(y))$  
is dense in $\mathcal G^{(0)}$ for all $y \in \mathcal G^{(0)}$. Let $\alpha= (\alpha_t)_{t\in \mathbb R}$ be a continuous
  one-parameter group of automorphisms on $C^*_r(\mathcal G)$ and
  assume that for some $\beta \in \mathbb R \backslash \{0\}$ there is
  a diagonal $\beta$-KMS weight for $\alpha$. Then $\alpha$ is diagonal, i.e. there is a continuous homomorphism $c : \mathcal G \to \mathbb R$ such
that
\begin{equation}\label{44}
\alpha_t(g)(\xi) = e^{it c(\xi)} g(\xi)
\end{equation}
for all $t\in \mathbb R$, all $g \in C_c(\mathcal G)$ and all $\xi \in
\mathcal G$. 
\end{thm}
\begin{proof} Combine Corollary \ref{77} and Proposition \ref{15}, using that in
  the presence of a single unit with trivial isotropy group the
  minimality of $\mathcal G$ is equivalent to the simplicity of
  $C^*_r(\mathcal G)$ by Corollary 2.18 in \cite{Th1}. 
\end{proof}

We can now put the pieces together for a \emph{proof of Theorem \ref{iff}:}
$1) \Rightarrow 3)$ follows from Proposition \ref{4}. That 3) is equivalent to
4) follows from a standard argument using that $\mathcal G^{(0)}$ is
totally disconnected. The implication $3) \Rightarrow
5)$ follows from Proposition \ref{15} and $5)
\Rightarrow 2)$ from Corollary \ref{75}. This gives the
equivalence of all five conditions since $2) \Rightarrow 1)$ is trivial.

\begin{example}\label{29} Let $\mathcal G = \mathbb F_2$ be the free group on two
  generators. Then $C^*_r(\mathbb F_2)$ is a simple $C^*$-algebra and
  $C_0(\mathcal G^{(0)}) = \mathbb C1$. Let $A = A^* \in C^*_r(\mathbb
  F_2)$ and set
$$
\alpha_t(a) = e^{itA}ae^{-itA} .
$$
Note that $\alpha_t$ acts trivially on  $C_0(\mathcal G^{(0)}) =
\mathbb C1$. Let $U_x, x \in \mathbb F_2$, be the canonical unitaries
generating $C^*_r(\mathbb F_2)$. Assume that there is a
homomorphism $c : \mathbb F_2 \to \mathbb R$ such that $\alpha_t(U_x)
= e^{itc(x)}U_x$ for all $t,x$. By differentiation this leads to the
conclusion that
$AU_x-U_xA = c(x)U_x$
and hence that
$U_x^*AU_x = A + c(x)1$. The last equation implies that the spectrum $\sigma(A)$ of $A$
satisfies $\sigma(A) = \sigma(A) + c(x)$, i.e. $c(x) = 0$. But then
$\alpha_t(U_x) = U_x$ for all $t,x$, i.e. $\alpha_t = \id$ for all $t
\in \mathbb R$. This implies by differentiation that
$A X = XA$ for all $X \in C^*_r(\mathbb F_2)$, i.e. $A$ is in the center of
$C^*_r(\mathbb F_2)$. So by choosing $A \notin \mathbb R 1$, we have
an example showing that Proposition \ref{15} does not always hold when there
are no units with trivial isotropy in $\mathcal G$. In relation to Theorem
\ref{iff} note that there are $\beta$-KMS weights for $\alpha$ for all
$\beta \in \mathbb R$. Indeed, when $\omega$ is the tracial state on
$C^*_r(\mathbb F_2)$, the functional
$$
C^*_r(\mathbb F_2) \ni a \mapsto \omega\left( e^{-\beta A}a\right)
$$
is a bounded $\beta$-KMS weight. Since condition 3) in Theorem
\ref{iff} holds while 5) does not, it follows that it is necessary, in
Theorem \ref{iff}, to assume the existence of a unit with trivial
isotropy group.

Similarly, by considering a disjoint union $\mathbb F_2 \sqcup \mathcal H$, where
$\mathcal H$ is an appropriate groupoid, it is easy to obtain examples
showing that the implication $4) \Rightarrow 1)$ in Theorem \ref{iff} fails in
general if $\mathcal G$ is not minimal.

\end{example}

\section{Applications to graph $C^*$-algebras}

In this section we apply the results obtained above to the study of
KMS weights on graph $C^*$-algebras. For this we first show how a
graph $C^*$-algebra can be realized as the groupoid $C^*$-algebra of a
locally defined
local homeomorphism as it was introduced by Renault in
\cite{Re2}. Recall that graph $C^*$-algebras were originally
introduced for row-finite graphs in \cite{KPRR} as the $C^*$-algebra of the
left-shift on the space of infinite paths in the graph. We show that in general, when the graph may have infinite emitters, its
$C^*$-algebra is still the groupoid $C^*$-algebra of a local
homeomorphism which is generally only defined on a dense open subset of a locally
compact Hausdorff space.

\subsection{The Renault groupoid of a local homeomorphism}

Let $X$ be a locally compact second countable Hausdorff space. Let $U
\subseteq X$ be an open subset and $\varphi : U \to X$ a local
homeomorphism, i.e. for every $u \in U$ there is an open subset $V
\subseteq U$ such that $u \in V$, $\varphi(V)$ is open and $\varphi : V \to \varphi(V)$ is
a homeomorphism. Set $\varphi^0 = \id_X$ (with domain $D(\varphi^0) = X$) and for $n \geq 1$, set 
$$
 D(\varphi^n) = U \cap \varphi^{-1}(U) \cap \varphi^{-2}(U) \cap \cdots \cap
\varphi^{-n+1}(U)
$$ 
and let $\varphi^n$
be the map 
$$
\varphi^n = \varphi \circ \varphi \circ \cdots \circ \varphi  : D(\varphi^n) \to X . 
$$  
Set
$$
\mathcal G_{\varphi} = \left\{ (x, n-m,y) \in X \times \mathbb Z \times X: \ \ x \in D\left(\varphi^n\right), \ y \in
  D\left(\varphi^m\right), \ \varphi^n(x) = \varphi^m(y) \right\} 
$$
which is a groupoid with product $(x,k,y) (y,l,z) = (x,k+l,z)$ and
inversion $(x,k,y)^{-1} = (y,-k,x)$. Sets of the form
$$
\left\{ (x,n-m,y) :  \ \varphi^n(x) = \varphi^m(y), \  x \in W, \ y \in V\right\} 
$$
for some open subsets $W \subseteq D(\varphi^n), \ V \subseteq
D(\varphi^m)$, constitute a basis for a topology in $\mathcal G_{\varphi}$ which
turns it into a locally compact second countable Hausdorff \'etale groupoid.

Let $F : X \to \mathbb R$ be a function which is continuous on $U$. We can then define
$c_F : \mathcal G_{\varphi} \to \mathbb R$ such that
$$
c_F(x,n-m,y) = \sum_{i=0}^{n} F\left(\varphi^i(x)\right) - \sum_{i=0}^m
F\left(\varphi^i(y)\right).
$$
Note that $c_F$ is a continuous homomorphism, and if $F' : X \to \mathbb R$ is a function which agrees with
$F$ on $U$, then $c_{F'} = c_F$.

\begin{prop}\label{31} Let $c : \mathcal G_{\varphi} \to \mathbb R$ be a
  continuous homomorphism. There is a map $F : X \to
  \mathbb R$ which is continuous on $U$ such that $c = c_F$.
\end{prop}
\begin{proof} Define $F : X \to \mathbb R$ such that
$$
F(x) = \begin{cases} c(x,1,\varphi(x)), &  \ x \in U \\ 0, & \ x \notin U .
\end{cases}
$$
It is straightforward to verify that $F$ is continuous on $U$
and that $c = c_F$.
\end{proof}

It follows that the continuous homomorphisms $\mathcal G \to \mathbb R$
are in bijective correspondence with the continuous maps $U \to
\mathbb R$.



A point $x \in X$ is \emph{aperiodic} under $\varphi$ when 
$$
x \in D(\varphi^n) \cap D(\varphi^m), \ \varphi^n(x) = \varphi^m(x) \ \Rightarrow
\ n = m.
$$
Under the identification of $X$
with the unit space of $\mathcal G_{\varphi}$ the aperiodic points are the
elements with trivial isotropy group. We can therefore combine
Proposition \ref{31} with Proposition \ref{15} to obtain the
following.

\begin{prop}\label{40} Let $X$ be a locally compact second countable Hausdorff space,
  $U \subseteq X$ an open subset and $\varphi : U \to X$ a local
  homeomorphism. Assume that $\alpha = (\alpha_t)_{t
      \in \mathbb R}$ is a
  continuous one-parameter group of automorphisms on $C^*_r(\mathcal
  G_{\varphi})$ such that
$$
\alpha_t(f) = f
$$
for all $f \in C_0\left(X\right) \subseteq C^*_r\left(\mathcal G_{\varphi}\right)$ and all $t \in
\mathbb R$. Assume also that the aperiodic points of $\varphi$ are dense in
$X$.

There is a continuous map $F : U \to \mathbb R$ such
that
$$
\alpha_t(g)(\xi) = e^{it c_F(\xi)} g(\xi)
$$
for all $t\in \mathbb R$, all $g \in C_c(\mathcal G_{\varphi})$ and all $\xi \in
\mathcal G_{\varphi}$. 

\end{prop}

For $n,m \in \mathbb N \cup \{0\}$, set
$$
\varphi^{-m}\left(\varphi^n(x)\right) = \begin{cases} \emptyset & \
  \text{when} \ x \notin D\left(\varphi^n\right) \\ \left\{ y \in
    D\left(\varphi^m\right) : \ \varphi^m (y) = \varphi^n(x) \right\} & \
  \text{when} \ x \in D(\varphi^n) . \end{cases}
$$
We say that $\varphi$ is \emph{minimal} when 
\begin{equation}\label{78}
\bigcup_{n,m \in \mathbb N \cup \{0\}} \varphi^{-m}\left(\varphi^n(x)\right)
\end{equation}
is dense in $X$ for all $x \in X$. Note that (\ref{78}) is the
orbit of $x$ under the action of $\mathcal G_{\varphi}$ on its unit
space. Thus $\varphi$ is minimal iff $\mathcal G_{\varphi}$ is.

\begin{prop}\label{41} Let $X$ be a locally compact second countable Hausdorff space,
  $U \subseteq X$ an open subset and $\varphi : U \to X$ a local
  homeomorphism. Assume that $\varphi$ is minimal and that there is at least one aperiodic point for
$\varphi$. Let $\alpha = (\alpha_t)_{t \in \mathbb R}$ be a
  continuous one-parameter group of automorphisms on $C^*_r(\mathcal
  G_{\varphi})$.

If, for some $\beta \neq 0$, there is a diagonal $\beta$-KMS weight for
$\alpha$, there is a continuous map $F : U \to \mathbb R$ such
that
\begin{equation}\label{42}
\alpha_t(g)(\xi) = e^{it c_F(\xi)} g(\xi)
\end{equation}
for all $t\in \mathbb R$, all $g \in C_c(\mathcal G_{\varphi})$ and all $\xi \in
\mathcal G_{\varphi}$. 

\end{prop}
\begin{proof} In view of Corollary \ref{77} and Proposition \ref{40}
  it suffices to observe that $ C^*_r\left(\mathcal G_{\varphi}\right)$ is
  simple under the present assumptions, cf. Proposition 2.5 in
  \cite{Re2}.
\end{proof}

\subsection{A local homeomorphism from an infinite graph}

Let $G$ be a directed graph with vertexes $V$ and edges $E$. We assume that $G$ is countable in the sense that $V$
and $E$ are both countable sets. We let $r$ and $s$ denote the maps
$r: E \to V$ and $s: E \to V$ which associate to an edge $e\in E$ its
target vertex $r(e)$ and source vertex $s(e)$, respectively. A vertex
$v$ is an
\emph{infinite emitter} when $s^{-1}(v)$ contains infinitely many
edges and a \emph{sink} when $s^{-1}(v)$ is empty. The union of
sinks and infinite emitters constitute a set which will be denoted by
$V_{\infty}$. The graph $C^*$-algebra
$C^*(G)$ is by definition the universal
$C^*$-algebra generated by a collection $S_e, e \in E$, of partial
isometries and a collection $P_v, v \in V$, of mutually orthogonal projections subject
to the conditions that
\begin{enumerate}
\item[1)] $S^*_eS_e = P_{r(e)}, \ \forall e \in E$,
\item[2)] $\sum_{e \in F} S_eS_e^* \leq P_v$ for every finite subset
  $F \subseteq s^{-1}(v)$ and all $v\in V$, and
\item[3)] $P_v = \sum_{e  \in s^{-1}(v)} S_eS_e^*, \ \forall v \in V
  \backslash V_{\infty}$.
\end{enumerate} 
Let $P_f(G)$ and
$P(G)$ denote the set of finite and infinite paths in $G$,
respectively. The range and source maps, $r$ and $s$, extend in the
natural way to $P_f(G)$; the source map also to $P(G)$. Set
$\Omega_G = P(G) \cup Q(G)$,
where 
$$
Q(G) = \left\{p \in P_f(G): \ r(p) \in V_{\infty} \right\} 
$$ 
is the set of finite paths that terminate at a vertex
in $V_{\infty}$. In particular, $V_{\infty} \subseteq Q(G)$ because
vertexes are considered to be finite paths of length $0$. For any $p
\in P_f(G)$, let $|p|$ denote the length of $p$. When $|p| \geq 1$, set
$$
Z(p) = \left\{ q \in \Omega_G: \ |q| \geq |p| , \ q_i = p_i, \ i = 1,2,
  \cdots, |p| \right\},
$$
and
$$
Z(v) = \left\{ q \in \Omega_G : \ s(q) = v\right\}
$$
when $v \in V$. When $\nu \in P_f(G)$ and $F$ is a finite subset of $P_f(G)$, set
\begin{equation}\label{a6}
Z_F(\nu) = Z(\nu) \backslash \left(\bigcup_{\mu \in F} Z(\mu)\right) .
\end{equation}
The sets $Z_F(\nu)$ form a basis of compact and open subsets for a locally compact Hausdorff
topology on $\Omega_G$. When $\mu \in P_f(G)$ and $  x \in \Omega_G$, we can define the
concatenation $\mu x \in \Omega_G $ in the obvious way when $r(\mu) =
s(x)$. The groupoid $\mathcal G_G$ consists of the
elements in $\Omega_G \times \mathbb Z \times \Omega_G$ of the form
$$
(\mu x, |\mu| - |\mu'|, \mu'x),
$$
for some $x\in \Omega_G$ and some $ \mu,\mu' \in P_f(G)$. The product
in $\mathcal{G}_{G}$ is defined by
$$
(\mu x, |\mu| - |\mu'|, \mu' x)(\nu y, |\nu| -|\nu'|, \nu' y) = (\mu
x, \ |\mu | + |\nu| - |\mu'| - |\nu'|, \nu' y),
$$ 
when $\mu' x = \nu y$, and the involution by $(\mu x, |\mu| - |\mu'|,
\mu'x)^{-1} = (\mu' x, |\mu'| - |\mu|, \mu x)$. To describe the
topology on $\mathcal{G}_{G}$, let $Z_{F}(\mu)$ and $Z_{F'}(\mu')$ be two
sets of the form (\ref{a6}) with $r(\mu) = r(\mu')$. The topology we
shall consider has as a basis the sets of the form
\begin{equation}\label{top}
\left\{ (\mu x, |\mu| - |\mu'|, \mu' x) : \ \mu x \in Z_F(\mu), \
  \mu'x \in Z_{F'}(\mu') \right\} .
\end{equation}
With this topology $\mathcal{G}_{G}$ becomes an \'etale locally compact
second countable Hausdorff groupoid and we can consider the reduced $C^*$-algebra $C^*_r(\mathcal{G}_{G})$ as in
\cite{Re1}. As shown by Paterson in
\cite{Pa} there is an isomorphism $C^*(G) \to
  C^*_r(\mathcal G_G)$ which sends $S_e$ to $1_e$, where $1_e$ is the
  characteristic function of the compact and open set
$$
\left\{ (ex, 1, r(e)x) : \ x \in \Omega_G \right\} \ \subseteq \
\mathcal{G}_{G},
$$  
and $P_v$ to $1_v$, where $1_v$ is the characteristic function of the
compact and open set
$$
\left\{ (vx,0,vx) \ : \ x \in \Omega_G \right\} \ \subseteq \ \mathcal{G}_{G}.
$$ 
In the following we use the identification $C^*(G) = C_r^*(\mathcal
G_G)$ and identify $\Omega_G$ with the unit space of $\mathcal G_G$ via
the embedding
$\Omega_G \ni x \ \mapsto \ ( x, 0,x)$.

Note that $\Omega_G \backslash V_{\infty}$ is an open subset of
$\Omega_G$ and that we can define a local homeomorphism
$$
\sigma : \Omega_G \backslash V_{\infty} \to \Omega_G
$$
such that $\sigma$ is the usual left shift on $P(G)$, defined such
that $\sigma(x)_i = x_{i+1}$, while $\sigma(e_1e_2 \cdots e_n)$ is defined as follows when $e_1e_2
\cdots e_n \in Q(G)$:
$$
\sigma(e_1e_2 \cdots e_n) = \begin{cases} e_2e_3\cdots e_n,  & \
  \text{when} \ n \geq 2 \\ r(e_1), & \ \text{when}  \ n =
  1. \end{cases}
$$
It is straightforward to check that there is an identification
$$
\mathcal G_G = \mathcal G_{\sigma} ,
$$
as topological groupoids.
In particular, it follows that any continuous function $F : \Omega_G
\backslash V_{\infty} \to \mathbb R$ defines a continuous homomorphism 
$c_F : \mathcal G_G \to \mathbb R$ such that
$$
c_F  (\mu x, |\mu| - |\mu'|, \mu'x) = \sum_{n=0}^{|\mu|}
F\left(\sigma^n(\mu x)\right) - \sum_{n=0}^{|\mu'|}
F\left(\sigma^n(\mu'x)\right) .
$$
To simplify notation the one-parameter group $\sigma^{c_F}$ defined
from $c_F$ will be denoted by $\sigma^F$. It follows from Proposition \ref{31} that every continuous
homomorphism $\mathcal G_G \to \mathbb R$ arises from a continuous
function $F : \Omega_G \backslash V_{\infty} \to \mathbb R$ as above. We can therefore formulate
Corollary \ref{65} in the following way for graph $C^*$-algebras.

\begin{thm}\label{80} Let $F : \Omega_G \backslash V_{\infty}  \to
  \mathbb R$ be a continuous function. 
There is a bijective correspondence
  between the $\beta$-KMS weights for $\sigma^{F}$ on
  $C^*(G)$ and the pairs $\left(\mu,
  \left\{\varphi_x\right\}_{x \in \Omega_G}\right)$, where
$\mu$ is a regular Borel measure on $\Omega_G$ and
$\left\{\varphi_x\right\}_{x \in \Omega_G}$ is a
$\mu$-measurable field of states $\varphi_x$ on
$C^*_r(({\mathcal G_G})^x_x)$ such that
\begin{enumerate}
\item[a)] $\mu$ is $e^{\beta F}$-conformal for $\sigma$,
\item[b)] $\varphi_x(u_g) = \varphi_{r(h)}\left(u_{hgh^{-1}}\right)$
    for $\mu$-almost every $x\in \Omega_G$ and all $g\in
    ({\mathcal G_G})^x_x, \ h \in ({\mathcal G_G)}_x$, and
\item[c)] $\varphi_x(u_g)  = 0$ for $\mu$-almost every $x \in
  \Omega_G$ and all $g \in (\mathcal G_G)^x_x \backslash c_F^{-1}(0)$.
\end{enumerate}
The $\beta$-KMS weight $\phi$ corresponding to the pair $\left(\mu,
  \left\{\varphi_x\right\}_{x \in \Omega_G}\right)$ has the
properties that $C_c\left({\mathcal G_G}\right) \subseteq \mathcal
M_{\phi}$ and
\begin{equation}\label{66}
\phi(f) = \int_{\Omega_G}  \sum_{g \in ({\mathcal G_G})^x_x}
f(g)\varphi_x(u_g) \ d\mu(x) 
\end{equation}
when $f \in C_c(\mathcal G_G)$.
\end{thm}

Similarly, for graph $C^*$-algebras our main result takes the
following form.

\begin{thm}\label{iff-graph} Let $G$ be a countable directed graph
  such that $C^*(G)$ is simple. Let $\alpha= (\alpha_t)_{t\in \mathbb R}$ be a continuous
  one-parameter group of automorphisms on $C^*(G)$ and
  assume that for some $\beta_0 \neq 0$ there is a $\beta_0$-KMS
  weight for
  $\alpha$.

The following are equivalent:
\begin{enumerate}
\item[1)]  There is a $\beta_1 \neq 0$ and a diagonal $\beta_1$-KMS
  weight for $\alpha$.
\item[2)] Whenever $\beta \neq 0$ and there is a $\beta$-KMS weight
  for $\alpha$, there is also a diagonal $\beta$-KMS weight for $\alpha$. 
\item[3)] $\alpha_t(f) = f$ for all $t \in \mathbb R$ and all $f\in
  C_0\left(\Omega_G\right)$.
\item[4)] There is a continuous function $F : \Omega_G \backslash
  V_{\infty} \to \mathbb R$ such that $\alpha = \sigma^{F}$.
\end{enumerate}
\end{thm}

It follows from Theorem \ref{80} (and Proposition \ref{31}) that all
KMS weights for a diagonal action on the $C^*$-algebra of a graph
without loops are diagonal. This is not true in general; not even for finite strongly connected graphs as shown in
\cite{Th4}. However, we can now show that it holds for strongly
connected graphs when the function $F$ has bounded local variation in
the a sense we now make precise. 

Let $v$ be a vertex in $G$ and set 
\begin{equation*}
\Var_{n,v} (F) = \sup_{x,y}  \left|\sum_{j=0}^{n-1}
    F\left(\sigma^j(x)\right) - \sum_{j=0}^{n-1}
    F\left(\sigma^j(y)\right) \right|
\end{equation*}
where we take the supremum over all pairs $x,y \in P(G)$ with the
property that $x_i=y_i, \ i = 1,2,\cdots, n$, and $s(x_1) = s(y_1)
=v$. The following condition (\ref{001}) should be compared with \emph{Bowen's condition} used by
Walters, cf. \cite{W}.

\begin{prop}\label{gaugeinv} Let $G$ be a countable directed graph such that
  $C^*(G)$ is simple and let $F : \Omega_G \backslash V_{\infty} \to
  \mathbb R$ be a continuous function such that for some
  vertex $v$, 
\begin{equation}\label{001}
\sup_n \Var_{n,v} (F) < \infty .
\end{equation}
Then every KMS weight for $\sigma^F$ is
diagonal.
\end{prop}
\begin{proof} The assumption that $C^*(G)$ is simple means that $G$ is
  cofinal in the sense used (e.g.) in \cite{Th3} and that every
  minimal loop
  in $G$ has an exit, cf. \cite{Sz}. It is easily seen that the set of
  vertexes $v$ for which (\ref{001}) holds
  is both hereditary and saturated. Under the present assumptions it
  will therefore hold for all $v$. Consider a $\beta$-KMS weight $\phi$
  and the pair $\left(\mu, \left\{\varphi_x\right\}_{x \in
      \Omega_G}\right)$ associated to it by Theorem \ref{80}. It
  suffices to show that the elements $x \in \Omega_G$ for which the
  isotropy group
  $\left(\mathcal G_G\right)^x_x$ is non-trivial is a null set with
  respect to $\mu$. The isotropy group of a point $x \in \Omega_G$ is
  non-trivial if and only if $x$ is an infinite pre-periodic path in
  $G$, and there are at most countably many such points. It suffices
  therefore to show that $\mu(\{x\}) = 0$ for any  infinite
  pre-periodic path $x$. There is an $m \in \mathbb N$ such that $x_0 = \sigma^m(x)$ is
  periodic. It
  follows from (\ref{conformal2}) that
$$
\mu(\{x\}) = e^{-\beta \sum_{j=0}^{m-1} F(\sigma^j(x))} \mu(\{x_0\}),
$$
so it suffices to show that $\mu(\{x_0\}) = 0$. 
Since $x_0$ is periodic there is a finite loop $\delta$ in
  $G$ such that $x_0 = \delta^{\infty}$, and since $G$ is cofinal and
  every loop in $G$ has an exit there is also a finite loop $\delta'$
  in $G$ such that $\delta' \nsubseteq x_0$ and $s(\delta') =
  s(\delta)$. By prolonging $\delta$ and $\delta'$ if necessary we may
  assume that the length of $\delta$ and $\delta'$ are the same, say
  $p$. For each $k \in \mathbb N$ set
$$
y_k = \delta^k \delta' x_0.
$$   
Since $x_0$ is $p$-periodic it follows
from (\ref{conformal2}) that 
$$
\mu\left(\{x_0\}\right) =  e^{-\beta \sum_{j=0}^{kp-1}
  F(\sigma^{j}(x_0))} \mu\left(\{x_0\}\right), 
$$ 
for all $k \in \mathbb N$, and the desired conclusion follows if $-\beta \sum_{j=0}^{kp-1}
  F(\sigma^{j}(x_0))$ is not zero for some $k$. Consider therefore now the
  case where $
-\beta \sum_{j=0}^{kp-1}
  F(\sigma^{j}(x_0)) = 0$ for all $k \in \mathbb N$. Since (\ref{001}) holds we find then that
\begin{equation}\label{con1}
 \left|\beta \sum_{j=0}^{kp-1}
    F\left(\sigma^j(y_k)\right)\right| =  \left|\beta \sum_{j=0}^{kp-1}
    F\left(\sigma^j(y_k)\right) - \beta \sum_{j=0}^{kp-1}
    F\left(\sigma^j(x_0)\right)\right| \leq |\beta|K 
\end{equation}
for all $k$, where $K = \sup_n \Var_{n,v} (F)$ and $v = s(\delta)$ is
the source of $\delta$. Now apply
(\ref{conformal2}) again to find that
$$
\mu \left(\left\{y_k\right\}\right) =  e^{-\beta \sum_{j=0}^{(k+1)p-1}
  F(\sigma^j(y_k))} \mu(\{x_0\}) .  
$$
Inserting (\ref{con1}) this leads to the conclusion that
$$
\mu \left(\left\{y_k\right\}\right) \geq e^{-|\beta|K} e^{-\beta
  \sum_{j=0}^{p-1} F\left(\sigma^j(z)\right)} \mu\left(\{x_0\}\right),
$$
where $z = \delta' x_0 = \sigma^{kp}(y_k)$. Since 
$$
 \sum_{k=1}^{\infty}
\mu\left(\left\{y_k\right\}\right) \leq \mu \left( Z(v)\right) < \infty,
$$
we conclude that $\mu\left(\{x_0\}\right) = 0$, as desired.    
  
\end{proof}

It follows from Proposition \ref{gaugeinv} that a generalized gauge
action on a graph $C^*$-algebra, considered for example in \cite{Th3},
where $F$ only depends on the first edge only has
gauge-invariant KMS weights, at least as long as the algebra is simple.

\begin{remark}\label{007} It should be emphasized that the conclusion
  in Proposition \ref{gaugeinv} does not hold without some condition
  on $F$. To see this observe that the example presented in \cite{Th4}
  shows that already for the canonical finite graph $G$ for which
  $C^*(G)$ is a copy of the Cuntz algebra $O_2$, namely the graph consisting of one vertex and
  two arrows, there are continuous non-negative functions $F : \Omega_G
  \to \mathbb R$ such that $\sigma^F$ admits non-diagonal KMS
  states. In the example from \cite{Th4} there is at least a single
  extremal KMS state which is diagonal, namely the extremal KMS-state
  corresponding to the lowest inverse temperature $\beta_0$. Here we want
  to indicate how to modify the example in \cite{Th4} to get an
  example where no \emph{extremal} KMS state is diagonal. The basis
  for this is a sequence $\{b_n\}_{n=1}^{\infty}$ of positive
  numbers with the following properties:
\begin{enumerate}
\item[a)] $b_{n} \geq b_{n+1} \ \forall n$,
\item[b)] $\lim_{n \to \infty} \frac{b_{n+1}}{b_n} = 1$,
\item[c)] $\sum_{n=1}^{\infty} b_n < 1$, and
\item[d)] $\sum_{n=1}^{\infty} b_n^s = \infty$ for all $s < 1$.
\end{enumerate}  
We leave the reader to verify the existence of such a sequence. Set
$a_1 = -\log b_1$ and $a_k = \log b_{k-1} - \log b_{k}, \ k \geq 2$, and
identify the infinite path space $\Omega_G$ with $\{0,1\}^{\mathbb N}$
by labelling the two arrows in $G$ by $0$ and $1$. Define then 
$T : \{0,1\}^{\mathbb N} \to \mathbb R$ such that $T\left((x_i)_{i=1}^{\infty}
\right) = a_k$ where $k = \min \{i : \ x_i = 0\}$ when
$(x_i)_{i=1}^{\infty} \neq 1^{\infty}$, and set $T(1^{\infty}) =
0$. (As in \cite{Th4} $1^{\infty}$ is the infinite string of $1$'s.) This
is a continuous non-negative function. By
using Theorem 2.2 in \cite{Th4} and arguing exactly as in Section 3 of
\cite{Th4}, but with the sequence $\left\{n^{-1}\right\}$ replaced by $\{a_n\}$,
it follows that there are $\beta$-KMS states for $\sigma^T$ if and
only if $\beta \geq 1$, and for each $\beta \geq 1$ the extremal KMS
states are parametrised by the circle, and none are diagonal. As guaranteed by Theorem \ref{iff-graph} there are for each $\beta \geq
1$ also one which is diagonal. As explained in \cite{Th4} it arises by
integrating the extremal ones with respect to Lebesgue measure on the circle.

\end{remark}

\end{document}